\documentclass{amsart}

\usepackage[sc]{mathpazo}
\usepackage{eulervm}
\usepackage{todonotes}

\usepackage[T1]{fontenc}
\usepackage[utf8]{inputenc}
\usepackage[english]{babel}
\usepackage{tikz}
\usepackage{amssymb}
\usepackage{mathrsfs}
\usepackage{float}
\usepackage{IEEEtrantools}
\usepackage{multirow}
\usepackage{amsthm}
\usepackage{amsmath}
\usepackage{amstext}
\usepackage{enumerate}

\usepackage{hyperref}

\newtheorem{theorem}{Theorem}
\newtheorem*{theorem*}{Theorem}
\newtheorem{lemma}{Lemma}
\newtheorem{proposition}{Proposition}
\newtheorem{corollary}{Corollary}
\theoremstyle{definition}
\newtheorem{definition}{Definition}

\theoremstyle{remark}
\newtheorem{remark}{Remark}
\newtheorem*{remark*}{Remark}

\title[Shellability of posets of labeled partitions]{Shellability of posets of labeled partitions and  arrangements defined by root systems}
\author{Emanuele Delucchi, Noriane Girard and Giovanni Paolini}

\newcommand{\lpns}{\Pi_{n,\Sigma}}
\newcommand{\E}{\mathcal E}

\newcommand{\precdot}{\prec \! \cdotp}
\newcommand{\succdot}{\,\cdotp \!\! \succ}

\newcommand{\one}{\hat{1}}
\newcommand{\zero}{\hat{0}}

\tikzstyle{treestyle}=[
  baseline=(current bounding box.north),
  level/.style={sibling distance=15mm},
  every node/.style={inner sep=2pt},
  level distance=12mm
]

\begin{document}

\maketitle

\vspace{-15pt}
\begin{abstract}
We prove that the posets of connected components of intersections of toric and elliptic arrangements defined by root systems  are EL-shellable and we compute their homotopy type.
Our method rests on Bibby's description of such posets by means of ``labeled partitions'': after giving an EL-labeling and counting homology chains for general posets of labeled partitions, we  obtain the stated results by considering the appropriate subposets.
\end{abstract}

\section{Introduction}

The focus of the theory of arrangements, which 
to date 
has been mainly on the study of linear subspaces, has recently broadened to arrangements of 
hypersurfaces in complex tori or in products of elliptic curves (called ``toric'' and ``elliptic'' arrangements).
This led to a renewed interest in the general combinatorial theory of topological dissections, which dates back to the Seventies. Here the poset of connected components of intersections of a family of submanifolds (which we will call \emph{poset of components} for short) has played a crucial role ever since Zaslavsky's seminal paper \cite{ZaT}.

In the classical case of arrangements of linear hyperplanes, the poset of components coincides with the poset of intersections, i.e., with the poset of flats of the associated matroid, and thus has the structure of a geometric lattice. In the case of toric arrangements, the poset of components has been studied for combinatorial purposes 
\cite{ERS,Law,M}
as well as in view of topological applications \cite{dCP2,CaDe}.
The literature on elliptic arrangements is still sparse; of particular interest for us is the work of Christin Bibby \cite{Bibby} who began a unified treatment of linear, toric and elliptic arrangements.

One of the main initial impulses for the study of linear arrangements came from the theory of Coxeter groups and root systems.
In fact, to every (finite) root system one can associate the linear arrangement of ``reflecting'' hyperplanes in the (contragradient) representation of the associated Coxeter group, and the rich combinatorial structure of Coxeter groups allows a particularly explicit description of the posets of intersections of such arrangements.
Bibby showed that this nice combinatorial description extends naturally to the toric and elliptic case (see Section \ref{ss_ars} for details).
Thus, arrangements defined by root systems appear to be a natural testbed for the study of arrangements also beyond the linear case.

\medskip

Shellability is a property of cell complexes first formalized in the context of boundaries of polytopes, where it gained celebrity due to its role in the solution of the upper bound conjecture for convex polytopes by McMullen \cite{McMullen}. The notion of shellability was subsequently extended to the context of simplicial (and regular CW-) complexes and is still widely studied nowadays due to its strong topological and algebraic implications. For instance: a shellable simplicial complex is homotopy equivalent to a wedge of spheres, and its Stanley-Reisner ring is Cohen-Macaulay. We refer to \cite{Kozlov} for a wider background on shellability, and we focus here on the case of partially ordered sets. A partially ordered set (or \emph{poset}) is called shellable if the simplicial complex of its totally ordered subsets (i.e., its {\em order complex}) is shellable. Shellability of posets was introduced by Bj\"orner, leading to the theory of lexicographic shellability that was first developed by Bj\"orner and Wachs \cite{BW1,BW2} and subsequently blossomed garnering the interest of a wide research community.

Posets of components of linear hyperplane arrangements --  and, more generally, geometric lattices -- are shellable. Indeed, geometric lattices can be characterized by their particularly nice shellability properties \cite{DH} and the number of spheres in the homotopy type of their order complex is an evaluation of the characteristic polynomial of the associated matroid. It is therefore natural to ask whether posets of components of toric and elliptic arrangements are shellable. 
 
%
%


In this paper we focus on the case of arrangements associated to root systems, where we can take advantage of Bibby's description of the posets of components as certain posets of labeled partitions. This leads us to introduce a two-parameter class of posets of labeled partitions which we prove to be EL-shellable. The posets of components of linear, toric or elliptic arrangements of every Weyl type are subposets of elements of our two-parameter class and the induced labelings are, in fact, EL-labelings. In particular, the homotopy type of those posets is that of a wedge of spheres; the number of these spheres can be interpreted in general as the number of \emph{$r$-flourishing, $q$-blooming trees} (see Definitions \ref{def_flourishing} and \ref{def_blooming}).
Then, we have a ``case-and-type free'' proof of the following.

\begin{theorem*}
Let $\mathcal A$ be a (linear, toric or elliptic) arrangement defined by a Weyl system and let $\mathcal C(\mathcal A)$ denote the associated poset of components. Then, $\mathcal C(\mathcal A)$ is EL-shellable. In particular the order complex of the poset $\overline{\mathcal C(\mathcal A)}$ (obtained by removing the minimum $\hat{0}$ and the maximum $\hat{1}$ if necessary) is homotopy equivalent to a wedge of spheres of dimension equal to the rank of the root system.
Closed formulas for the number of spheres are given in Table \ref{table:num-spheres}.
\end{theorem*}

\begin{table}[htbp]
  \begin{center}
  \begin{tabular}{c|c|c|c}
    Type & Linear & Toric & Elliptic \\\hline\hline
    $A_n$ & $n!$ & $n!$ & $n!$ \\\hline
    $B_n$ & \multirow{2}{*}{$(2n -1)!!$} & $(2n-3) !! (n-1)$ & \footnotesize{[...]} \\ \cline{1-1}\cline{3-4}
    $C_n$ & & $(2n-1) !!$ & $(2n+1)!!$ \\\hline
    $D_n$ & $(2n-3) !! (n-1)$ & $(2n-5)!! (n^2-3n+3)$ & \footnotesize{[...]}
  \end{tabular}
  \end{center}
  \bigskip
  \caption{Closed formulas for the number of spheres.
  In the two missing cases, the number of spheres can be obtained by setting $m=4$ in the formulas of Theorem \ref{thm:decreasing-chains}. See Remark \ref{rem_C1} for a discussion of how some of these values could also be obtained as evaluations of characteristic or Tutte-type polynomials.}
  \label{table:num-spheres}
\end{table}


This paper is structured as follows.
In Section \ref{sec_pre} we introduce the poset $\lpns$ of \emph{$\Sigma$-labeled partitions} of the set $\{1,\overline{1},\ldots,n,\overline{n}\}$, where $\Sigma$ is any finite set. We then describe the four classes of subposets of $\lpns$ which, following Bibby \cite{Bibby}, are isomorphic to the posets of components of linear ($\vert\Sigma\vert=1$), toric ($\vert\Sigma\vert=2$) and elliptic ($\vert\Sigma\vert=4$) arrangements defined by the four infinite Weyl families.
In Section \ref{sec_she} we recall the notion of shellability and review some technical results on edge-lexicographic shellings.
In Section \ref{sec_res} we prove the shellability of posets $\lpns$ by constructing an EL-labeling. Then we prove that this labeling induces EL-labelings on all the above-mentioned subposets corresponding to posets of components of arrangements defined by root systems.
In particular, the (reduced) order complexes of said posets are homotopy equivalent of a wedge of spheres; the homotopy type of such posets is hence determined by the number of these spheres. The closing Section \ref{sec_sph} tackles this enumeration problem: we give general formulas (for any $\Sigma$) and deduce closed expressions for (almost) all cases of arrangements defined by root systems.

\subsection*{Acknowledgements} We acknowledge support by the Swiss National Science Foundation Professorship grant PP00P2\_150552/1. The research presented here took place during the second author's work on her master's thesis, written under the first author's supervision, while the third author was visiting University of Fribourg on a Swiss Government Excellence Fellowship. We also thank Masahiko Yoshinaga and Nhat Tan Tran for their feedback after a conference talk presentation of part of this work.

\section{Preliminaries}
\label{sec_pre}

In this section we introduce the main characters and review some relevant results from the literature. On matters concerning posets we adopt the terminology from \cite{Stanley}.

\subsection{The poset of $\Sigma$-labeled partitions of $[[n]]$}
A {\em partition} of a given finite set $S$ is a collection $\pi=\{B_1,\ldots,B_l\}$ of disjoint sets, called ``blocks'', such that $\bigcup_i B_i = S$. We will signify this by writing 
\[ \pi=B_1 \mid \cdots \mid B_l. \]
Let now $\Sigma$ be a finite set (of ``signs''). A {\em partition of $S$ labeled by $\Sigma$} is a partition $\pi$ of $S$ together with a subset $T \subseteq \pi$ and an injection $f:T \to \Sigma$. The blocks in $T$ are called {\em signed blocks} of $\pi$ and those of $\pi \setminus T$ are called {\em unsigned blocks} of $\pi$. 

For any given integer $n\in \mathbb N$ let
\[ 
[n]:= \{1,\ldots,n\};\quad\quad
[[n]] := \{1,\overline{1},2,\overline{2},\ldots,n,\overline{n}\}. \]
Notice that the set $[[n]]$ carries a natural involution $\overline{\cdot}: [[n]] \to [[n]]$ defined as $i\mapsto \overline{i}$, $\overline{i}\mapsto i$ for all $i=1,\ldots,n$. In general, given a subset $X\subseteq [[n]]$ we write
\[ \overline{X}:=\{\overline{x} \mid x\in X\}. \]


The finest partition of $[[n]]$ is that in which every block is a singleton.
We will denote it by
\[ \zero := 1 \mid \overline{1} \mid 2 \mid \overline{2} \mid \cdots \mid n \mid \overline n. \]

\begin{definition}
  Let $\Sigma$ be a set (of ``signs''). We say that $\pi$ is a {\em $\Sigma$-labeled partition} of $[[n]]$ if $\pi$ is a partition of $[[n]]$ labeled by $\Sigma$ such that 
  \begin{itemize}
  \item for every $R \in \pi$, $\overline{R} \in \pi$, and 
  \item $S=\overline{S}$ if and only if $S\in \pi$ is signed. 
  \end{itemize}
  We will write this as follows:
  \begin{equation}\label{eq:pi}
  \pi = S_{\sigma_1} \mid \cdots \mid S_{\sigma_l} \mid R_1 \mid \overline{R_1}\vert \cdots \mid R_k \mid \overline{R_k}.
  \end{equation}
  Here the signed blocks are denoted by $S$, and carry the sign $\sigma_i \in \Sigma$ as an index. The unsigned blocks are denoted by $R_i$. We call $\Sigma(\pi):=\{\sigma_1, \ldots, \sigma_l \}$ the \emph{sign} of $\pi$.
\end{definition}

Given a $\Sigma$-labeled partition as in Equation \eqref{eq:pi} and a sign $\sigma \in \Sigma(\pi)$, we write $\pi_{\sigma} := S_{\sigma}$.

\begin{definition}
  Let $\pi$ be a $\Sigma$-labeled partition of $[[n]]$ written as in Equation \eqref{eq:pi} and let
  \[ \psi :=S'_{\tau_1} \mid  \ldots \mid S'_{\tau_m} \mid R'_1 \mid  \overline{R'_1} \mid \cdots \mid R'_q \mid \overline{R'_q}. \]
  We say that $\pi$ is a {\em refinement} of $\psi$ if:
  \begin{itemize}
   \item for every $\sigma \in \Sigma(\pi)$ there exists some $\tau \in \Sigma(\psi)$ such that $\pi_\sigma \subseteq \psi_\tau$;
   \item for every $i\in[k]$ there are two blocks of $\psi$, say $T$ and $\overline T$, such that $R_i\subseteq T$ and $\overline{R_i} \subseteq \overline T$ (notice that $T$ and $\overline T$ can be equal if $T=\overline T = \psi_{\tau_j}$ for some $j\in [m]$).
  \end{itemize}
  We write $\pi \preceq \psi$ if $\pi$ is a refinement of $\psi$. Notice that this defines a partial order relation on the set of all $\Sigma$-labeled partitions of $[[n]]$.
\end{definition}

\begin{remark}
  In the situation of the above definition, if there is no $\Sigma$-labeled partition $\phi$ of $[[n]]$ such that $\pi \prec \phi \prec \psi$ we say that $\psi$ {\em covers} $\pi$ and write $\pi  \precdot \psi$.
\end{remark}

\begin{definition}
We write $\lpns$ for the poset of all $\Sigma$-labeled partitions of $[[n]]$, with a new element $\one$ added to the top, partially ordered by $\preceq$.
Therefore, the element $\one$ covers all the maximal $\Sigma$-labeled partitions.
\end{definition}

\begin{remark}
The poset $\lpns$ is ranked, with rank function given by: $\text{rk}(\pi) = n-k$ for $\pi\neq\one$ (where $k$ is as in Equation \eqref{eq:pi}), and $\text{rk}(\one) = n+1$.
\end{remark}

\subsection{Arrangements associated to root systems and their posets of components}
\label{ss_ars}

Let $G$ denote a dimension one complex algebraic group (i.e., $G$ is either $\mathbb C$, $\mathbb C^*$, or an elliptic curve $E$). We consider arrangements of subvarieties in $G^d$ defined by any set of elements $a_1,\ldots,a_l$ in a lattice (free abelian group) $L$ of rank $d$ as follows: for every $i=1,\ldots,l$ let
\[
H_i:=\{\varphi\in \operatorname{Hom}(L,G) \mid a_i\in \ker\varphi\},
\]
a subset of $\operatorname{Hom}(L,G)\simeq G^d$. The (linear, toric or elliptic) arrangement associated to $a_1,\ldots,a_l$ (and to $G=\mathbb C$, $\mathbb C^*$ or $E$, respectively) is the set
\[
\mathcal A :=\{H_1,\ldots,H_l\}.
\]

A {\em component} of $\mathcal A$ is any connected component of an intersection $\cap \mathcal B$ for some $\mathcal B \subseteq \mathcal A$. The {\em poset of components} of $\mathcal A$ is denoted by $\mathcal C (\mathcal A)$. 

\begin{remark}[Characteristic polynomials]
\label{rem_CPTP} Recall that the {\em characteristic polynomial} of a finite, ranked and bounded-below poset $P$ is
$$
\chi_P(t):=\sum_{p\in P} \mu_P(\hat 0, p) t^{\operatorname{rk}(P)- \operatorname{rk}(p)}.
$$
If $\mathcal A$ is an arrangement, the poset $\mathcal C (\mathcal A)$ is by nature bounded below and ranked by codimension, and its characteristic polynomial 
can be computed as an evaluation of Tutte-type polynomials. In fact, $\mathcal A \subseteq G^d$ can be seen as the quotient of an arrangement of hyperplanes $\mathcal A ^\upharpoonright\subseteq \mathbb C^d$ by a group $\Gamma$ of translations. If $\mathcal A$ is itself a linear arrangement, $\Gamma = \{\operatorname{id}\}$. If $\mathcal A$ is a toric (or elliptic) arrangement, $\mathcal A ^\upharpoonright$ is an infinite, periodic arrangement of hyperplanes, and $\Gamma \simeq \mathbb Z^d$ (or $\mathbb Z^{2d}$). In any case we have an action of the group $\Gamma$ on the semimatroid defined by $\mathcal A^\upharpoonright$, to which in \cite{DR} is associated a two variable {\em Tutte polynomial} $T_{\mathcal A}(x,y)$. Also, recall \cite[Theorem F]{DR}: 
\[
  \chi_{\mathcal C(\mathcal A)} (t) = (-1)^d T_{\mathcal A}(1-t,0).
\]
In particular, when $\mathcal A$ is a linear arrangement $T_{\mathcal A} (x,y)$ equals the classical Tutte polynomial of the associated matroid; when $\mathcal A$ is a toric arrangement, then $T_{\mathcal A} (x,y)$ is Moci's arithmetic Tutte polynomial \cite{M}. 
\end{remark}

A distinguished class of examples arises from root systems of type $A$, $B$, $C$ or $D$. To any such root system we associate a (linear, toric or elliptic) arrangement by taking $L$ equal to the coroot lattice and letting $a_1,\ldots,a_l$ be a set of positive roots.

\begin{theorem}[Barcelo-Ihrig {\cite[Theorem 4.1]{barcelo1999lattices}}, Bibby {\cite[Theorem 3.3]{Bibby}}] 
\label{thm_bibby}
The posets of components of arrangements associated to root systems have the following form, where $|\Sigma|=1$ in the linear case, $|\Sigma|=2$ in the toric case and $|\Sigma|=4$ in the elliptic case.

\medskip
\begin{tabular}{lp{0.8\textwidth}}
  {\bf Type $A_n$:} & $\mathcal C (\mathcal A)$ is isomorphic to the subposet of $\lpns\setminus\{\one\}$ consisting of all elements of the form $R_1 \mid \overline{R_1} \mid \cdots \mid R_k\mid \overline{R_k}$ where $R_1\mid \cdots\mid R_k$ is a partition of $[n]$. Thus, in all cases $\mathcal C(\mathcal A)$ is isomorphic to the classical lattice of all partitions of $[n]$ ordered by refinement. \\[0.1em]
  
  {\bf Type $B_n$:} & Fix some distinguished element $\overline\sigma \in \Sigma$.
  $\mathcal C (\mathcal A)$ is isomorphic to the subposet of $\lpns\setminus\{\one\}$ consisting of all elements $x\in \lpns\setminus\{\one\}$ such that $\vert x_{\sigma} \vert\neq 2$ whenever $\sigma\neq \overline\sigma$. \\[0.1em]
  
  {\bf Type $C_n$:} & $\mathcal C (\mathcal A)$ is isomorphic to $\lpns \setminus\{\one\}$. \\[0.1em]
  
  {\bf Type $D_n$:} & $\mathcal C (\mathcal A)$ is isomorphic to the subposet of $\lpns\setminus\{\one\}$ consisting of all elements $x\in \lpns\setminus\{\one\}$ such that $\vert x_{\sigma} \vert\neq 2$ for all $\sigma\in \Sigma$.
\end{tabular}
\end{theorem}

Our results will apply to the following more general families of posets.

\begin{definition}\label{def_gen} For any $n\in \mathbb N$ and any finite set of signs $\Sigma$ with a distinguished element $\overline\sigma \in \Sigma$, the poset of $\Sigma$-signed partitions of type $A_n$, $B_n$, $C_n$ or $D_n$ is the subposet of $\lpns$ satisfying the corresponding condition in Theorem \ref{thm_bibby}.
\end{definition}

\begin{remark}
  All the subposets in Theorem \ref{thm_bibby} and Definition \ref{def_gen} are ranked, with rank function induced by $\lpns$.
\end{remark}

\begin{remark}\label{rem_CP}
When $\mathcal A$ is an arrangement of hyperplanes, the characteristic polynomial of $\mathcal C (\mathcal A)$
 is known classically 
 \cite[Definition 2.52, Theorem 4.137, Corollary 6.62]{OrlikTerao}).  In the toric case, Ardila-Castillo-Henley \cite{ACH} gave formulas for the arithmetic Tutte polynomials of the associated arrangements, from which in principle the characteristic polynomial can be computed. In the elliptic case no explicit formula is known to us.
 \end{remark}

\section{Shellability of posets and subposets}\label{sec_she}

\subsection{Shellings}
We refer to \cite{BW1} for basics on the notion of shellings of simplicial complexes. Here we only recall that, from a topological point of view, a shellable simplicial complex has the homotopy type of a wedge of spheres. When the complex at hand is the order complex of a poset, the following theorem gives an interpretation of the sphere count in terms of the poset's characteristic polynomial. 

\begin{theorem}\label{thm_CP}
Let $P$ be a bounded and ranked poset and suppose that $\Delta(\overline P)$ is shellable, where $\overline P = P \setminus \{ \zero, \one \}$.
Then $\Delta(\overline P)$ is a wedge of $(-1)^d\chi_P(0)$ spheres of dimension $\operatorname{rk}(P)-2$, where $\chi_P$ denotes the characteristic polynomial of $P$.
\end{theorem}
\begin{proof}
If $P$ is ranked, the complex $\Delta(\overline P)$ is of pure dimension $d:=\operatorname{rk}(P)-2$. Moreover, any shellable  complex of pure dimension $d$ decomposes in a wedge of spheres all of which have dimension $d$.
Therefore, the Euler characteristic of $\Delta(\overline P)$ is
$
1 + (-1)^d k
$ 
where $k$ is the number of spheres in the wedge. A theorem of P.\ Hall \cite[Theorem 3.8.6]{Stanley} states that the Euler characteristic of $\Delta(\overline P)$ equals $ \mu_P(\hat 0,\hat 1) + 1$. But by definition of the characteristic polynomial one sees $ \mu_P(\hat 0,\hat 1)=\chi_P(0)$. Therefore, $\chi_P(0)+1=1 + (-1)^d k$, thus $k=(-1)^d\chi_P(0)$.
\end{proof}

\subsection{EL-labelings}

We follow \cite[Section 5]{BW1} for the definition of EL-labelings.
Let $P$ be a bounded poset, where the top and bottom elements are denoted by $\one$ and $\zero$, respectively.
In addition, denote by $\E(P)$ the set of edges of the Hasse diagram of $P$, and by $\E_{xy}$ the edge between $x$ and $y$ whenever $x \precdot y$.

An \emph{edge labeling} of $P$ is a map $\lambda\colon \E(P) \to \Lambda$, where $\Lambda$ is some poset.
Given an edge labeling $\lambda$, each maximal chain $c = (x \precdot z_1 \precdot \dots \precdot z_t \precdot y)$ between any two elements $x \preceq y$ has an associated word
\[ \lambda(c) = \lambda(\E_{x z_1}) \, \lambda(\E_{z_1 z_2}) \,\cdots\, \lambda(\E_{z_t y}). \]
We say that the chain $c$ is \emph{increasing} if the associated word $\lambda(c)$ is strictly increasing, and \emph{decreasing} if the associated word is weakly decreasing.
Maximal chains in a fixed interval $[x,y]\subseteq P$ can be compared ``lexicographically'' (i.e.\ by using the lexicographic order on the corresponding words).

\begin{definition}
  Let $P$ be a bounded poset.
  An \emph{edge-lexicographical labeling} (or simply \emph{EL-labeling}) of $P$ is an edge labeling such that in each closed interval $[x,y] \subseteq P$ there is a unique increasing maximal chain which lexicographically precedes all other maximal chains of $[x,y]$.
\end{definition}

The main motivation for introducing EL-labelings of posets is given by the following theorem.

\begin{theorem}[{\cite[Theorem 5.8]{BW1}}]
  Let $P$ be a bounded poset with an EL-labeling.
  Then the lexicographic order of the maximal chains of $P$ is a shelling of the order complex $\Delta(P)$.
  Moreover, the corresponding order of the maximal chains of $\overline P$ is a shelling of $\Delta(\overline P)$.
\end{theorem}

The shelling induced by an EL-labeling is called an \emph{EL-shelling}. A bounded poset that admits an EL-labeling is said to be \emph{EL-shellable}.

\begin{remark}\label{rem_spherecount}
If $P$ is an EL-shellable poset, the homotopy type of $\Delta(\overline{P})$ is that of a wedge of spheres indexed over the maximal chains of $P$ with decreasing label. More precisely, if $D_P$ is the set of maximal chains of $P$ with decreasing labels,
\[
  \Delta(\overline{P}) \simeq \bigvee_{c\in D_P} S^{\vert c \vert -2}.
\]
Notice that, if $P$ is ranked, all maximal chains have the same cardinality.
\end{remark}

\subsection{Products of posets}\label{sec_prod}
Let $P_1$ and $P_2$ be bounded posets that admit EL-labelings $\lambda_1\colon \E(P_1)\to \Lambda_1$ and $\lambda_2\colon \E(P_2) \to \Lambda_2$, respectively. Assume that $\Lambda_1$ and $\Lambda_2$ are disjoint and totally ordered.
Let $\lambda\colon \E(P_1 \times P_2) \to \Lambda_1 \cup \Lambda_2$ be the edge labeling of $P_1\times P_2$ defined as follows:
\begin{IEEEeqnarray*}{lCl}
  \lambda(\E_{(a,b)(c,b)}) &=& \lambda_1(\E_{ac}), \\
  \lambda(\E_{(a,b)(a,d)}) &=& \lambda_2(\E_{bd}).
\end{IEEEeqnarray*}

\begin{theorem}[{\cite[Proposition 10.15]{BW2}}]\label{thm_product}
  Fix any shuffle of the total orders on $\Lambda_1$ and $\Lambda_2$, to get a total order of $\Lambda_1 \cup \Lambda_2$.
  Then the product edge labeling $\lambda$ defined above is an EL-labeling.
\end{theorem}

\subsection{A criterion for subposets} 
When $P$ is a shellable bounded poset, certain subposets $Q\subseteq P$ are also shellable through some induced labeling.
A general criterion is given by \cite[Theorem 10.2]{BW2}.
Here we only state and prove the following simple and more specific criterion.

\begin{lemma}\label{lem_sub}
  Let $P$ be a bounded and ranked poset, with an EL-labeling $\lambda$.
  Let $Q\subseteq P$ be a ranked subposet of $P$ containing $\zero$ and $\one$, with rank function given by restricting the rank function of $P$.
  Then, if for all $x \preceq y$ in $Q$ the unique increasing maximal chain in $[x,y] \subseteq P$ is also contained in $Q$, the edge labeling $\lambda|_{\E(Q)}$ is an EL-labeling of $Q$.
\end{lemma}

\begin{proof}
  The hypothesis on the ranks ensures that $\E(Q) \subseteq \E(P)$, so it makes sense to restrict $\lambda$ to $\E(Q)$.
  
  Let $x \preceq y$ be elements of $Q$.
  Denote by $[x,y]_Q$ the interval bounded by $x$ and $y$ in $Q$, and by $[x,y]_P$ the interval bounded by $x$ and $y$ in $P$.
  Since $\lambda$ is an EL-labeling of $P$, there exists a unique increasing maximal chain $c$ in $[x,y]_P$: this chain belongs to $[x,y]_Q$ by hypothesis.
  Maximal chains of $[x,y]_Q$ are also maximal chains of $[x,y]_P$ by the hypotesis on the ranks.
  Since $c$ lexicographically precedes all other maximal chains of $[x,y]_P$, in particular it precedes all other maximal chains of $[x,y]_Q$.
\end{proof}

%

\section{Shellability of posets of $\Sigma$-labeled partitions}\label{sec_res}

\subsection{EL-labeling for the poset of $\Sigma$-labeled partitions of $[[n]]$.}

For a block $R$ of some $\Sigma$-labeled partition, define the \emph{representative} of $R$ as the minimum element $i \in [n]$ such that $i\in R$ or $\bar i \in R$.
Denote by $r(R)$ the representative of $R$.
Notice that $R$ and $\overline R$ share the same representative, and that the representative of $R$ does not necessarily belong to $R$.

\smallskip

We call a block $R \subseteq [[n]]$ \emph{normalized} if the representative of $R$ belongs to $R$. For instance, $\{2, \bar 4, 5 \}$ is normalized and $\{\bar 2, 4, 5\}$ is not.
Notice that exactly one of $R$ and $\overline R$ is normalized, whenever $R$ is an unsigned block.

\smallskip

An edge $\E_{xy}$ of the Hasse diagram of $\lpns$ with $y \neq \one$ is called:
\begin{itemize}
 \item \emph{of sign $\sigma$} if $x_\sigma \subsetneq y_\sigma$ for a $\sigma\in \Sigma$;
 \item \emph{coherent} if $R \cup R' \in y$ for some normalized unsigned blocks $R, R' \in x$;
 \item \emph{non-coherent} if $R \cup \overline{R'} \in y$ for some normalized unsigned blocks $R, R' \in x$.
\end{itemize}
If $y \neq \one$, the edge $\E_{xy}$ is of exactly one of these three types.
We also say that $\E_{xy}$ is \emph{unsigned} if it is either coherent or non-coherent, and that it is \emph{signed} otherwise.

\begin{definition}\label{def_lambda}
  Given a total ordering $<$ of $\Sigma$ define the following edge labeling $\lambda$ of $\lpns$.
  \begin{itemize}
    \item Let $\E_{xy}$ be an unsigned edge. Let $R,R'\in x$ such that $R\cup R' \in y$, as above, and let $i$ and $j$ be the representatives of $R$ and $R'$. Then:
    \[ \lambda(\E_{xy}) =
    \begin{cases}
      (0, \max(i,j)) & \text{if $\E_{xy}$ is coherent}; \\
      (2, \min(i,j)) & \text{if $\E_{xy}$ is non-coherent}.
    \end{cases}
    \]
    
    \item Let $\E_{xy}$ an edge of sign $\sigma$. Then:
    \[ \lambda(\E_{xy}) =
    \begin{cases}
      (1,\vert \Sigma(x)_{\leq \sigma}\vert) & \text{if $\sigma \in \Sigma(x)$ }; \\
      (1,\vert \Sigma_{\leq \sigma}\cup  \Sigma(x) \vert ) & \text{otherwise}.
    \end{cases}
    \]

    \item For $x \precdot \one$, let
    \[
      \lambda(\E_{x \one}) = (1,2).
    \]
  \end{itemize}
  Labels are ordered lexicographically.
  \label{def:EL-labeling}
\end{definition}

\begin{theorem}\label{thm_main}
  The labeling $\lambda$ of Definition \ref{def:EL-labeling} is an EL-labeling of $\lpns$.
\end{theorem}

\begin{proof}
  The proof is divided into five parts.
  
  \begin{enumerate}
    \item {\em Intervals $[x,z]$ that contain only coherent edges.} Here $x$ and $z$ must have the same signed part. Thus, every such interval is isomorphic to an interval of the lattice of standard partitions of $[n]\setminus \bigcup_{\sigma\in \Sigma} x_\sigma$. The isomorphism maps an element $x_{\sigma_1} \mid \ldots \mid x_{\sigma_l} \mid R_1 \mid \overline{R_1} \mid \ldots \mid R_k \mid \overline{R_k} \in [x,z]$ to $\widetilde{R}_1 \mid \ldots \mid \widetilde{R}_k$ where, for all $i$, $\widetilde{R}_i:= (R_i\cup \overline{R_i}) \cap[n]$. This isomorphism maps the labeling $\lambda$ to one of the standard EL-labelings of the partition lattice \cite[Example 2.9]{Bjorner1980}. In particular there is an unique increasing maximal chain from $x$ to $z$.

    \item {\em Intervals of the form $[x, \one]$.}
    Recursively define a maximal chain $c_{x\one}$ in $[x,\one]$ as follows.
    
    \noindent For $x \precdot \one$, define $c_{x\one}=(x \precdot \one)$.
    Here notice that $\lambda(\E_{x\one})=(1,2)$.
    
    \noindent Let $x \prec \one$ be an element of $\lpns$ not covered by $\one$.
    \begin{itemize}
    \item Case 1: there exist unsigned normalized blocks $R \neq R'$ in $x$. Among all such pairs $(R,R')$, choose the (only) one for which $(r(R), r(R'))$ is lexicographically least. Let $z$ be the $\Sigma$-labeled partition obtained from $x$ by replacing $R \mid \overline R \mid R' \mid \overline{R'}$ with $R\cup R'\mid \overline{R} \cup \overline{R'}$. 

    Then set $c_{x\one} = (x \precdot c_{z\one})$.
    Notice that $\mathcal{E}_{xz}$ is coherent and $\lambda(\mathcal{E}_{xz})= (0,r(R'))$.
    
    \item Case 2: no such pair $(R,R')$ exists. Since $x$ is not covered by $\one$, there exists a unique normalized unsigned block $R$ in $x$. Let $z$ be the $\Sigma$-labeled partition obtained from $x$ by labeling $R \cup \overline{R}$ by $\sigma$ where $\sigma$ is:
    \begin{itemize}
      \item $\min_<\Sigma(x)$ if $\Sigma(x) \neq \emptyset $;
      \item $\min_<\Sigma$ if $\Sigma(x) = \emptyset $. 
    \end{itemize} 
    Then set $c_{x\one}=x\precdot c_{z\one}$. 
    Notice that $\mathcal{E}_{xz}$ is of sign $\sigma$ and $\lambda(\mathcal{E}_{xz})=(1,1)$.
    Also, notice that $z \precdot \one$, so $c_{x\one} = (x \precdot z \precdot \one)$.
    \end{itemize}
    
    \noindent We are going to prove that $c_{x\one}$ is the unique increasing maximal chain in $[x,\one]$, and that it is lexicographically least among all the maximal chains in $[x,\one]$.
    
    \begin{itemize}
      \item[$\diamond$] By construction, $c_{x\one}$ is increasing.
      
      \item[$\diamond$] Let $c$ be any increasing maximal chain in $[x,\one]$. We want to show that $c=c_{x\one}$.
      Assume that $x$ is not covered by $\one$ (otherwise the chain is trivial).

      Since the last label of $c$ is $(1,2)$, $c$ does not contain any non-coherent edges, and it contains exactly one signed edge. So $c$ contains: first, coherent edges, with labels of the form $(0, \ast)$; then, one signed edge labeled $(1,1)$; finally, one edge ending in $\one$, labeled $(1,2)$.

      Therefore there exists some $z \in c$ such that $c \cap [x,z]$ only consists of coherent edges and $c \cap [z,\one]$ consists in one signed edge and one edge ending in $\one$. Here $z$ is uniquely determined by $x$, and $c \cap [z,\one]$ is uniquely determined by the increasing property of $c$.

      The interval $[x,z]$ contains only coherent edges, thus by Part (1) it contains a unique increasing maximal chain which must coincide with $c \cap [x,z]$. Therefore there is an unique increasing maximal chain in $[x,\one]$ and $c=c_{x\one}$.     
      \item[$\diamond$] In both Case $1$ and Case $2$ above, $\lambda(\mathcal{E}_{xz})< \lambda(\mathcal{E}_{xz'})$ for every $z'\neq z$ that covers $x$ in $[x,\one]$. Then $c_{x\one}$ is lexicographically least among the maximal chains in $[x,\one]$.
    \end{itemize}

    \item {\em Intervals $[x,y]$ where $y$ is of the form $y= R \mid \overline{R}$.}
    Recursively define a maximal chain $c_{xy}$ in $[x,y]$ as follows.
    
    \noindent For $x=y$, set $c_{yy} = y$.
    Assume then $x \prec y$.
    \begin{itemize}
      \item Case 1: there exist unsigned normalized blocks $T \neq T'$ in $x$ such that $T\cup T'$ is contained in a block of $y$. Choose the unique pair $(T,T')$ of blocks of $x$ such that $(r(T),r(T'))$ is lexicographically least. Let $z$ be the partition obtained from $x$ by replacing $T\mid \overline{T} \mid T' \mid \overline{T'}$ with $T\cup T' \mid \overline{T} \cup \overline{T'}$.
      Then set $c_{xy}= (x \precdot c_{zy})$.
      Notice that $\mathcal{E}_{xz}$ is coherent and $\lambda(\mathcal{E}_{xz})= (0,r(T'))$.
      
      \item Case 2: no such pair $(T,T')$ exists.
      Then $x \precdot y$ and there exist unsigned normalized blocks $Q \neq Q'$ in $x$ such that $Q\cup \overline{Q'}$ coincides with $R$ or $\overline{R}$.
      Then set $c_{xy}= (x \precdot y)$.
      Notice that $\mathcal{E}_{xy}$ is non-coherent and $\lambda(\mathcal{E}_{xy})=(2,r)$, where $r= \min ( r(Q), r(Q') )$.
    \end{itemize}
    
    \noindent We are going to prove that $c_{xy}$ is the unique increasing maximal chain in $[x,y]$, and that it is lexicographically least among all the maximal chains in $[x,y]$.
    \begin{itemize}
      \item[$\diamond$] By construction, $c_{xy}$ is increasing.
      
      \item[$\diamond$] Let $c$ be any increasing maximal chain in $[x,y]$. 
      We want to show that $c=c_{xy}$.

      The chains in $[x,y]$ have no signed edge. Since the label of coherent edges is smaller than the label of non-coherent edges, coherent edges precede non-coherent edges along $c$.
      So there exists an element $z \in c$ such that $c \cap [x,z]$ only consists of coherent edges, and $c \cap [z,y]$ only consists of non-coherent edges.

      If the last edge of $c$ is non-coherent, then its label must be $(2, r(R))$. This is also the minimum possible label for a non-coherent edge in the interval $[x,y]$.
      Therefore, since $c$ is increasing, it contains at most one non-coherent edge.

      Thus $z$ is uniquely determined by $x$ and $y$: if $R$ is the union of the normalized blocks of $x$, then $c$ cannot end with a non-coherent edge, so $z=y$; otherwise, $z = Q\mid \overline{Q} \mid Q' \mid \overline{Q'}$ where $Q \cup \overline{Q'} = R$ and both $Q$ and $Q'$ are unions of normalized blocks of $x$.
      Also $c\cap [z,y]$ is uniquely determined.

      Consider now the interval $[x,z]$. 
      This interval contains only coherent edges, thus by Part (1) it contains a unique increasing maximal chain which must coincide with $c \cap [x,z]$. Therefore there is an unique increasing maximal chain in $[x,\one]$ and $c=c_{x\one}$.


      \item[$\diamond$] In both Case $1$ and Case $2$ above, $\lambda(\E_{xz}) < \lambda(\E_{xz'})$ for any $z' \neq z$ that covers $x$ in the interval $[x,y]$. Then  $c_{xy}$ is lexicographically least among the maximal chains in $[x,y]$.
    \end{itemize}

    \item {\em Intervals $[x,y]$ where $y$ has only signed blocks or (unsigned) singleton blocks.}
    Recursively define a maximal chain $c_{xy}$ in $[x,y]$ as follows.
    
    \noindent For $x=y$, set $c_{yy} = y$. Assume then $x \prec y$.
    \begin{itemize}
      \item Case 1: there exist unsigned normalized blocks $R \neq R'$ of $x$ such that $R\cup R'$ is contained in some block of $y$.
      Among all such pairs $(R,R')$, choose the (only) one for which $(r(R), r(R'))$ is lexicographically minimal.
      Let $z$ be the $\Sigma$-labeled partition obtained from $x$ by replacing $R \mid \overline R \mid R' \mid \overline{R'}$ with $R\cup R'\mid \overline{R} \cup \overline{R'}$.
      Set $c_{xy} = (x \precdot c_{zy})$.
      Notice that $\E_{xz}$ is coherent, so $\lambda(\E_{xz}) = (0, r(R'))$.
      
      \item Case 2: no such pair $(R,R')$ exists.
      Then there are at most $\vert \Sigma(y) \vert$ elements $z\in [x,y]$ that cover $x$, since for every sign $\sigma \in \Sigma(y)$ there is at most one element $z$ such that $\mathcal{E}_{xz}$ is of sign $\sigma$.
      If possible, choose $z$ such that the sign of $\E_{xz}$ equals $\min \,\{ \sigma \in \Sigma(x) \: \mid \: x_{\sigma} \subsetneq y_{\sigma} \}$. Otherwise choose $z$ such that the sign of $\E_{xz}$ equals $\min \, \{ \sigma \in \Sigma(y) \: \mid \: x_{\sigma} \subsetneq y_{\sigma} \}$.
      Set $c_{xy} = (x \precdot c_{zy})$.
      Notice that $\E_{xz}$ is a signed edge, and $\lambda(\E_{xz})$ is of the form $(1,\ast)$.
    \end{itemize}

    \noindent We are going to prove that $c_{xy}$ is the unique increasing maximal chain in $[x,y]$, and that it is lexicographically least among all the maximal chains in $[x,y]$.
    
    \begin{itemize}
    \item[$\diamond$] By construction, $c_{xy}$ is increasing.
    \item[$\diamond$] Let $c$ be any increasing maximal chain in $[x,y]$. We want to show that $c=c_{xy}$.
      
    Since $y$ contains no unsigned block, the last edge of $c$ must be signed. In particular it must have a label of the form $(1, \ast)$.
    Then an increasing chain in $[x,y]$ can only contain coherent edges and signed edges. 
    In addition, coherent edges precede signed edges along $c$.
    Therefore there exists an element $z \in c$ such that $c \cap [x,z]$ only consists of coherent edges and $c \cap [z,y]$ only consists of signed edges.

    It is clear that every maximal chain in $[x,y]$ contains at least
    \[ t:= \vert \{ \sigma \in \Sigma(y) \: \mid \: x_\sigma \subsetneq y_\sigma \} \vert \]
    signed edges.
    Suppose that a maximal chain $d$ of $[x,y]$ contains more than $t$ signed edges. So $d$ contains at least two edges of the same sign $\tau$, and one of the following cases occurs. (If $d$ has more than two edges of sign $\tau$, consider the first two such edges.)
    
    \begin{itemize}
    \item[$\ast$] $\tau \in \Sigma(x)$ and the two edges of sign $\tau$ are separated only by edges of sign in $\Sigma(x)$. Then $d$ has two edges with the same label and therefore it is not increasing.
     
    \item[$\ast$] $\tau \in \Sigma(x)$ and the two edges of sign $\tau$ are separated only by edges of sign $\sigma \in \Sigma(y) \setminus \Sigma(x)$ such that $\tau < \sigma$. Then the two edges of sign $\tau$ have the same label, and $d$ is not increasing.
    
    \item[$\ast$] $\tau \in \Sigma(x)$ and the two edges of sign $\tau$ are separated by some edge of sign $\sigma \in \Sigma(y) \setminus \Sigma(x)$ such that $\sigma < \tau$. Then the last such edge has a label greater than or equal to the label of the second edge of sign $\tau$, and $d$ is not increasing.
    
    \item[$\ast$] $\tau \in \Sigma(y) \setminus \Sigma(x)$. Then the second edge of sign $\tau$ has a label smaller than or equal to the label of the first edge of sign $\tau$ and $d$ is not increasing.
    \end{itemize}
    
    In any case, $d$ is not increasing. So every increasing maximal chain in $[x,y]$ contains exactly $t$ signed edges, and all these edges are of different sign.
      
    Therefore $z$ is uniquely determined by $x$ and $y$.
    The only way for $c\cap [z,y]$ to be increasing is to contain first all the edges of sign $\sigma \in \Sigma(x)$ in increasing order and then all the edges of sign $\sigma \in \Sigma(y) \setminus \Sigma(x)$ in increasing order.
    Then $c\cap [z,y]$ is also uniquely determined.
      
    Consider now the interval $[x,z]$, which contains only coherent edges. By Part (1), such interval contains a unique increasing maximal chain which must coincide with $c \cap [x,z]$.
    In particular, there is an unique increasing maximal chain in $[x,y]$ and $c=c_{xy}$. 
    
%

    \item[$\diamond$] In both Case $1$ and Case $2$ above, $\lambda(\E_{xz}) < \lambda(\E_{xz'})$ for any $z' \neq z$ that covers $x$ in the interval $[x,y]$. Then  $c_{xy}$ is lexicographically least among the maximal chains in $[x,y]$.
    \end{itemize}
  
  \item {\em General intervals $[x,z]$ with $z\neq \one$.} Given any set of disjoint (signed or unsigned) blocks $X_1,\ldots,X_t$ let us write $\{\{X_1, \cdots , X_t\}\}$ for the $\Sigma$-labeled partition whose blocks are $X_1,\ldots,X_t$ and a singleton block for every element of $[[n]]\setminus (X_1\cup\ldots\cup X_t)$.
  
  Write
  \[ z = z_{\sigma_1} \mid \cdots \mid z_{\sigma_l} \mid R_1 \mid \overline{R_1} \mid \cdots \mid R_k \mid \overline{R_k}, \]
  let $x_0:= \{\{ B\in x : B\subseteq z_{\sigma}\textrm{ for some }\sigma\in \Sigma(z) \}\}$ be the $\Sigma$-labeled partition whose nonsingleton blocks are the nonsingleton blocks of $x$ contained in some signed block of $z$, and define the following intervals:
  \begin{IEEEeqnarray*}{lCl}
    J(0) &:=& \left[x_0, \, \{\{ z_{\sigma_1}, \cdots, z_{\sigma_l} \}\} \right] \\
    J(i) &:=& \left[ \{\{ B\in x : B\subseteq R_i\cup \overline{R_i} \}\},\, \{\{ R_i, \overline{R_i} \}\} \right] \text{ for } i=1,\ldots,k.
  \end{IEEEeqnarray*}
  There is a poset isomorphism
  \[ J(0) \times\cdots\times J(k) \to [x,z] \]
  mapping every $(k+1)$-tuple of partitions on the left-hand side to their common refinement.
  
  In Parts (3)-(4) we have proved that the labeling $\lambda$ on each $J(i)$ is an EL-labeling. If we let $\Lambda_i$ be the set of labels used in $J(i)$ we see that $\Lambda_i\cap \Lambda_j =\emptyset$ whenever $i\neq j$. Thus, the lexicographic order on $\Lambda_0\cup\cdots\cup \Lambda_{k}$ is a shuffle of the lexicographic orderings of the $\Lambda_i$s. The edge-labeling induced on $[x,z]$ as in Section \ref{sec_prod} equals 
  $\lambda$ which, by Theorem \ref{thm_product}, is then an EL-labeling of $[x,z]$.
   
  \end{enumerate}
\end{proof}

\subsection{Arrangements associated to root systems}

\begin{theorem}\label{thm_typeEL}
For any set of signs $\Sigma$ with a distinguished element $\overline\sigma \in \Sigma$, the poset of $\Sigma$-labeled partitions of type $A_n$, $B_n$, $C_n$ and $D_n$ is EL-shellable.
\end{theorem}
\begin{proof}
 For type $C_n$ the claim is a direct consequence of Theorem \ref{thm_main}. For type $A_n$ the poset of components is isomorphic to a classical partition lattice which, being a geometric lattice, is EL-shellable.

Now let $\mathcal P$ denote a poset of $\Sigma$-labeled partitions of type $B_n$ or $D_n$ with a top element $\hat{1}$ added. Then, $\mathcal P$ is a ranked subposets of $\lpns$ with rank function induced by that of the ambient poset. In view of Lemma \ref{lem_sub} we consider $x\preceq z$ in $\mathcal P$ and write
\[ x\precdot x_1 \precdot \cdots \precdot
\underbrace{x_{i-1} \precdot x_{i}}_{\E_{i}} \precdot \cdots \precdot z \]
for the unique increasing chain in $\lpns$. We have to show that this chain is contained in $\mathcal P$. By way of contradiction let $i$ be minimal such that $x_i\not\in \mathcal P$. The only way this can happen is if there is $\sigma\in \Sigma(z) \setminus \Sigma(x_{i-1})$ such that $\vert(x_i)_{\sigma}\vert =2$ (in particular, $\E_i$ is a signed edge). This implies that $z_\sigma \neq \emptyset$, but since $z\in \mathcal P$ it must be $z_\sigma \supsetneq (x_i)_\sigma$. In particular there must be another edge $\E_j$ of sign $\sigma$ with $j > i$. But then $\lambda(\E_j) \leq \lambda(\E_i)$: a contradiction.
\end{proof}

\begin{corollary}\label{thm_rootEL}
Posets of components of (linear, toric, elliptic) arrangements associated to root systems are EL-shellable. In particular, $\Delta (\overline{\mathcal C (\mathcal A)})$ is homotopy equivalent to a wedge of $\vert T_{\mathcal A} (1,0) \vert$ spheres.
\end{corollary}
\begin{proof}
The first claim is immediate from Theorem \ref{thm_typeEL} via the classification given in Theorem \ref{thm_bibby}. The second claim follows from Theorem \ref{thm_CP} and Remark \ref{rem_CPTP}.
\end{proof}

\begin{remark}\label{rem_C1}
Corollary \ref{thm_rootEL} expresses the homotopy type of posets of components of arrangements associated to root systems in terms of the arrangement's Tutte (or characteristic, cf.\ Remark \ref{rem_CPTP}) polynomial. This would allow us already to fill in the leftmost column of Table \ref{table:num-spheres} and, in principle, also the middle column (see Remark \ref{rem_CP}).
In the elliptic case, no explicit form for the Tutte polynomials of the corresponding group action is known to us. Therefore, in the following section we address the problem of the homotopy type of these posets from another (unified) perspective, using Remark \ref{rem_spherecount}. 
\end{remark}

\section{The homotopy type of posets of labeled partitions}\label{sec_sph}

In order to determine the homotopy type of posets of labeled partitions we have to count maximal decreasing chains in our EL-labeled posets (see Remark \ref{rem_spherecount}).
To this end, we first introduce \emph{increasing ordered trees}.
Such trees appeared in the literature under various names (e.g.\ heap-ordered trees \cite{chen1994heap}, ordered trees with no inversions \cite{gessel1995enumeration}, simple drawings of rooted plane trees \cite{klazar1997twelve}).
\begin{definition}
An \emph{increasing ordered tree} is a rooted tree, with nodes labeled $0,1,\dots,n$ (for some $n\geq 0$), such that:
\begin{itemize}
  \item each path from the root to any leaf has increasing labels (in particular, the root has label $0$);
  \item for each node, a total order of its children is specified.
\end{itemize}
\end{definition}
We will sometimes use a different (ordered) set of labels (but the root will have the minimal label).
See Figure \ref{fig:trees} for a picture of the three different increasing ordered trees on $3$ nodes.

\begin{figure}[htbp]
  \begin{tikzpicture}[treestyle]
  \node [circle,draw] (1) {0}
    child {node [circle,draw] (2) {1}
      child {node [circle,draw] (3) {2}
      }
    };
  \end{tikzpicture}
  \qquad
  \begin{tikzpicture}[treestyle]
  \node [circle,draw] (1) {0}
    child {node [circle,draw] (2) {1}
    }
    child {node [circle,draw] (3) {2}
  };
  \end{tikzpicture}
  \qquad
  \begin{tikzpicture}[treestyle]
  \node [circle,draw] (1) {0}
    child {node [circle,draw] (3) {2}
    }
    child {node [circle,draw] (2) {1}
  };
  \end{tikzpicture}
  \caption{All different increasing ordered trees on $3$ nodes. The second and the third tree differ in the total order of the children of the root.}
  \label{fig:trees}
\end{figure}
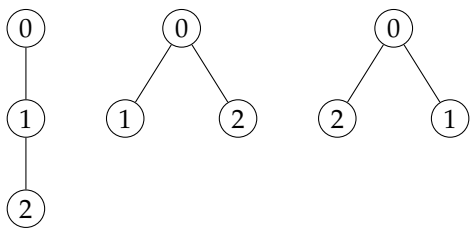

In general, the number of increasing ordered trees on $n+1$ nodes is $(2n-1)!!$.
This is a classical result which appeared several times in the literature  \cite{chen1994heap,gessel1995enumeration,klazar1997twelve}.
A simple proof by induction is as follows: given an increasing ordered tree on $n$ nodes (labeled $0,\dots,n-1$), there are $2n-1$ ways to append an additional node labeled $n$, to form an increasing ordered tree on $n+1$ nodes.

We are now ready to state and prove the general result about the number of decreasing maximal chains.

\begin{theorem}
  \label{thm:decreasing-chains}
  The number of decreasing maximal chains from $\zero$ to $\one$, in the subposet of $\lpns$ of type $B_n$, $C_n$, or $D_n$, is
  \begin{equation}
    \sum_{\pi \,=\, \pi_1 \mid \dots \mid \pi_k \,\in\, \Pi_n}
    f(\pi) \cdot (2|\pi_1|-3)!! \cdots (2|\pi_k|-3)!!.
    \label{eq:decreasing-chains}
  \end{equation}
  Here $\Pi_n$ is the poset of standard partitions of $[n]$, and $f\colon \Pi_n \to \mathbb N$ depends on the type ($B$, $C$, or $D$) and on $m=|\Sigma|$.
  In order to define $f(\pi)$, let us denote by $k$ the number of blocks of $\pi$, and by $s$ the number of nonsingleton blocks of $\pi$.
  
\medskip
\noindent\begin{tabular}{ll}
  {\bf Type $C_n$:} & $f(\pi) = \frac{(k+m-2)!}{(m-2)!}$.
  \, For $m=1$, set $f(\pi)=0$. \\[1em]
  
  {\bf Type $D_n$:} & $f(\pi) =
    \begin{cases}
      \sum_{r=2}^{\min(m,s)} \binom{m}{r} \frac{(k-2)!}{(r-2)!} \frac{s!}{(s-r)!} & \text{if $\pi \neq [n]$}; \\
      m-1 & \text{if $\pi = [n]$}.
    \end{cases}$ \\[2em]
  
  {\bf Type $B_n$:} & $f(\pi) =
    \begin{cases}
      \arraycolsep=1.4pt\def\arraystretch{1.6}
      \begin{array}{l}
	\sum_{r=2}^{\min(m,\,s+1)} \binom{m-1}{r-1} \frac{(k-2)!}{(r-2)!} (k-r+1)\frac{s!}{(s-r+1)!} \\
	\quad\quad\quad\quad\quad\!\! +  \sum_{r=2}^{\min(m-1,\,s)} \binom{m-1}{r} \frac{(k-2)!}{(r-2)!} \frac{s!}{(s-r)!}
      \end{array}
      & \text{if $\pi \neq [n]$}; \\[2.4em]
      m-1 & \text{if $\pi = [n]$}.
    \end{cases}$
\end{tabular}
\end{theorem}

\begin{proof}
  Let $\mathcal{D}$ be the set of decreasing maximal chains from $\zero$ to $\one$.
  When read bottom-to-top, chains $c\in\mathcal{D}$ consist of:
  \begin{itemize}
   \item a sequence of non-coherent edges, labeled $(2,*)$;
   \item then, a sequence of signed edges, labeled $(1,*)$;
   \item finally, one last edge labeled $(1,2)$.
  \end{itemize}
  Define a projection $\eta \colon \mathcal{D} \to \Pi_{n,\emptyset} \subseteq \lpns$, mapping a chain $c\in\mathcal{D}$ to the (unique) element $z \in c$ which comes after all the non-coherent edges and before all the signed edges.
  In addition, let $\rho \colon \Pi_{n,\emptyset} \to \Pi_n$ be the natural order-preserving projection that maps a partition $z = R_1 \mid \overline{R_1} \mid \dots \mid R_k \mid \overline{R_k} \in \Pi_{n,\emptyset}$ to the partition $\pi = \pi_1 \mid \dots \mid \pi_k \in \Pi_n$ with blocks $\pi_i = (R_i \cup \overline{R_i}) \cap [n]$.
  
  Label the edges of $\Pi_n$ following the ``non-coherent'' rule: if $\alpha \precdot \beta$ in $\Pi_n$, and $\beta$ is obtained from $\alpha$ by merging the two blocks $\alpha_i$ and $\alpha_j$, then label $\E_{\alpha\beta}$ by $\min(\alpha_i \cup \alpha_j)$.
  This is not an EL-labeling of $\Pi_n$, but it is (the second component of) the image of our EL-labeling of $\lpns$ through the map $\rho$, restricted to non-coherent edges of $\Pi_{n,\emptyset} \subseteq \lpns$.
  
  Both maps $\eta$ and $\rho$ are clearly surjective.
  We claim that, for each partition $\pi = \pi_1 \mid \dots \mid \pi_k \in\Pi_n$,
  \[ \big| (\rho\circ\eta)^{-1}(\pi) \big| = f(\pi) \cdot (2|\pi_1|-3)!! \cdots (2|\pi_k|-3)!!. \]
  We prove this through the following steps.
  \begin{enumerate}[1.]
    \item The number of decreasing maximal chains of $[\zero, \pi] \subseteq \Pi_n$ is
     \[ (2|\pi_1|-3)!! \cdots (2|\pi_k|-3)!!. \]
    \item In the preimage under $\rho$ of every decreasing maximal chain of $[\zero, \pi] \subseteq \Pi_n$, there is exactly one decreasing chain of $\Pi_{n,\emptyset}$ with all non-coherent edges.
    \item For every $z \in \rho^{-1}(\pi)$, the number of decreasing maximal chains of $[z, \one]$ without unsigned edges is $f(\pi)$.
    Notice that the interval $[z, \one]$ is different in each subposet of $\lpns$, depending on the type ($B$, $C$, or $D$); the definition of $f$ changes accordingly.
  \end{enumerate}
  These three steps prove the above claim, which in turn completes the proof of this theorem.
  Let us prove each of the three step.
  \begin{enumerate}[1.]
      \item The interval $[\zero, \pi]$ is a product of intervals $[\zero, \pi_1], \dots, [\zero, \pi_k]$, and the labeling splits accordingly.
      Without loss of generality, we can then assume that $\pi$ consists of a single block $\pi_1 \cong \{ 1, \dots, \ell \}$.
      Decreasing maximal chains $c$ of $[\zero, \pi]$ are in one-to-one correspondence with increasing ordered trees $T$ on nodes $\{1,\dots, \ell\}$: a node $j$ is a direct child of a node $i<j$ in $T$ if and only if the chain $c$ features a merge of blocks $R$ and $R'$ with $\min(R)=i$ and $\min(R')=j$; the set of direct children of a node is ordered according to the sequence of the corresponding merges in $c$.
      The number of such trees is $(2\ell-3)!!$, as discussed at the beginning of this section.
      
      \item Given a decreasing maximal chain $c$ of $[\zero, \pi] \subseteq \Pi_n$, a lift to $\Pi_{n,\emptyset}$ can be constructed explicitly by induction, starting from $\zero$ and parsing the chain in increasing order.
      In fact, for every edge $\alpha \precdot \beta$ in $\Pi_n$ and for every choice of $x \in \rho^{-1}(\alpha)$, there are two ways to lift $\beta$ to some $y \succdot x$, and in exactly one of these two cases the resulting edge $\E_{xy}$ is non-coherent.
      
      \item Such a chain is uniquely determined by an ordering of the blocks $\pi_1,\dots,\pi_k$ of $\pi$ (which will determine the order of the merges), and by a decreasing sequence of $k$ labels
      \begin{equation}
	(1,m) \geq (1,a_1) \geq \dots (1,a_k) \geq (1,2)
	\label{eq:decreasing-sequence}
      \end{equation}
      to be assigned to the edges (the sequence of labels uniquely determines the sequence of signs).
      Any ordering of the blocks of $\pi$ and any decreasing sequence of labels gives rise to a valid chain, provided that (in types $B_n$ and $D_n$) no unadmissible block $i \bar i_\sigma$ is created. Given a decreasing sequence of labels, we call such an ordering of the blocks of $\pi$ a \emph{valid ordering}.
      In type $C_n$, any ordering of the blocks of $\pi$ is always valid, regardless of the sequence of labels.
      Notice that for $m=1$ there are no decreasing sequences of labels: in this case $f(\pi)=0$.
      
      \begin{itemize}
	\item Let us first consider type $C_n$, so that all signed blocks are admissible.
	Then there are $k!$ possible orderings of the blocks, and $\binom{k+m-2}{k}$ decreasing sequences of $k$ labels, so we get
	\[ f(\pi) = k! \cdot \binom{k+m-2}{k} = \frac{(k+m-2)!}{(m-2)!}. \]
      
	\item Consider now type $D_n$.
	The sequence of signs associated to a decreasing sequence of labels as in \eqref{eq:decreasing-sequence} contains exactly $r$ different signs if and only if $a_r \geq r$ and $a_{r+1} \leq r$.
	Therefore, for a fixed $r \in \{1,\dots,m\}$, the sequences of labels which give rise to exactly $r$ different signs are those of the form
	\[ \qquad\quad (1,m) \geq (1,a_1) \geq \dots \geq (1,a_r) \geq (1,r) \geq (1,a_{r+1})\geq \dots \geq (1, a_k) \geq (1,2). \]
	The number of such sequences is $\binom{m}{r} \binom{k-2}{r-2}$ for $r\geq 2$.
	For $r=1$, there are $m-1$ sequences if $k=1$, and $0$ otherwise.
	
	If $k=1$ (i.e.\ $\pi = [n]$), exactly one sign is used and we get $f([n]) = m-1$. Suppose from now on $k\geq 2$, i.e.\ $\pi\neq [n]$.
	
	We have to compute the number of valid orderings of the $k$ blocks of $\pi$.
	Recall that $s$ is the number of nonsingleton blocks of $\pi$.
	Each time a sign appears for the first time in the chain, a nonsingleton block must be signed.
	Then, for a fixed number $r \in\{2,\dots, m\}$ of signs appearing in the chain, the number of valid orderings of the $k$ blocks is $s\left(s-1\right)\cdots \allowbreak \left(s-r+1\right) \allowbreak \cdot (k-r)!$.
	Therefore
	\begin{IEEEeqnarray*}{rCl}
	  f(\pi) \; &=& \sum_{r=2}^{\min(m,s)} \binom{m}{r} \binom{k-2}{r-2} \frac{s!}{(s-r)!} (k-r)! \\
	  & =& \sum_{r=2}^{\min(m,s)} \binom{m}{r} \frac{(k-2)!}{(r-2)!} \frac{s!}{(s-r)!}.
	\end{IEEEeqnarray*}
	
	\item Finally consider type $B_n$, where we have a distinguished sign $\overline\sigma$ that can admit singleton blocks.
	It is convenient to assume that $\sigma_1$ is the maximal element of $\Sigma$ with respect to the total order of Definition \ref{def:EL-labeling}.
	Then a sequence of labels $(1,m) \geq (1,a_1) \geq \dots \geq (1,a_k) \geq (1,2)$ gives rise to a sequence of signs containing $\sigma_1$ if and only if $a_1 = m$.
	Both cases $a_1=m$ and $a_1 \leq m-1$ can be treated as in type $D_n$. \qedhere
      \end{itemize}
  \end{enumerate}
\end{proof}

\subsection{Linear arrangements}

In the case of linear arrangements ($m=1$), the number of maximal decreasing chains is $0$. This happens because $\Pi_{n,\{\sigma\}} \setminus \{\one\}$ still has a unique maximal element $[n]_\sigma$. Therefore one is interested in counting the decreasing chains from $\zero$ to $[n]_\sigma$.

\begin{theorem}
  The number of maximal decreasing chains from $\zero$ to $[n]_\sigma$, in the subposet of $\Pi_{n, \{\sigma\}}$ of type $B_n$, $C_n$, or $D_n$, is given by \eqref{eq:decreasing-chains} with the following definition of $f\colon \Pi_n \to \mathbb N$.
  
  \medskip
  \begin{tabular}{ll}
    {\bf Type $B_n = C_n$:} & $f(\pi) = k!$. \\[0.1em]
    {\bf Type $D_n$:} & $f(\pi) = s \cdot (k-1)!$.
  \end{tabular}
\end{theorem}

\begin{proof}
  It is completely analogous to the proof of Theorem \ref{thm:decreasing-chains}.
\end{proof}

We want to give a graph-theoretic interpretation of Formula \eqref{eq:decreasing-chains} for these two definitions of $f\colon \Pi_n \to \mathbb N$, in order to obtain the (well-known) closed formulas of Table \ref{table:num-spheres} for linear types $B_n = C_n$ and $D_n$.

\begin{proposition}
  Formula \eqref{eq:decreasing-chains} with $f(\pi) = k!$ counts the increasing ordered trees on $n+1$ nodes.
  Therefore the number of decreasing chains from $\zero$ to $[n]_\sigma$ for the linear type $B_n=C_n$ is $(2n-1)!!$.
  \label{prop:decreasing-chains-linear-B}
\end{proposition}

\begin{proof}
  Consider an increasing ordered tree on $n+1$ nodes, labeled $0,1,\dots,n$.
  The set of subtrees corresponding to the children of the root is a partition of $\{1,\dots,n\}$.
  For a fixed partition $\pi \in \Pi_n$, the possible trees that induce such partition can be recovered as follows: order the blocks of $\pi$ in $k!$ ways; for each block $\pi_i$, construct an increasing ordered tree with nodes labeled by elements of $\pi_i$, in $(2|\pi_i|-3)!!$ ways.
  Summing over the partitions $\pi\in\Pi_n$, Formula \eqref{eq:decreasing-chains} is obtained.
\end{proof}

An analogous interpretation can be given for the linear type $D_n$, provided that we introduce the following class of increasing ordered trees.

\begin{definition}\label{def_flourishing}
  Let $r\geq 0$ be an integer.
  An increasing ordered tree is \emph{$r$-flourishing} if the root has at least $r$ children, and none of the first $r$ children of the root is a leaf.
  See Figure \ref{fig:flourishing-trees} for some examples.
\end{definition}

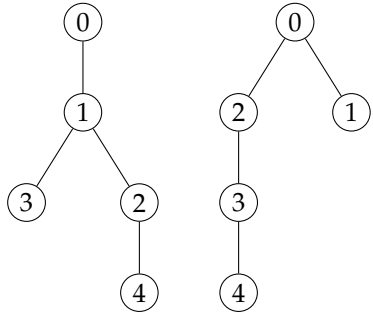
\begin{figure}[htbp]
  \begin{tikzpicture}[treestyle]
  \node [circle,draw] (1) {0}
    child {node [circle,draw] (2) {1}
      child {node [circle,draw] (4) {3}
      }
      child {node [circle,draw] (3) {2}
	child {node [circle,draw] (5) {4}
	}
      }
    };
  \end{tikzpicture}
  \qquad
  \begin{tikzpicture}[treestyle]
  \node [circle,draw] (1) {0}
    child {node [circle,draw] (3) {2}
      child {node [circle,draw] (4) {3}
	child {node [circle,draw] (5) {4}
	}
      }
    }
    child {node [circle,draw] (2) {1}
  };
  \end{tikzpicture}
  \caption{Both these trees are $1$-flourishing but not $2$-flourishing.}
  \label{fig:flourishing-trees}
\end{figure}

\begin{lemma}
  The number of $r$-flourishing trees on $n+1$ nodes is given by
  \[ (2n-2r-1)!! \cdot (n-r)(n-r-1)\cdots(n-2r+1) \]
  for $n \geq 2r$, and $0$ otherwise.
  \label{lemma:flourishing-trees}
\end{lemma}

\begin{proof}
  Let $\nu(n,r)$ be the number of $r$-flourishing trees on $n+1$ nodes.
  We are going to prove the following recursive relation:
  \[ \nu(n,r) = \nu(n-1,r) \cdot (2n-r-1) + \nu(n-2,r-1) \cdot (n-1) r. \]
  \begin{itemize}
    \item The first summand counts the $r$-flourishing trees obtained by appending an additional node $n$ to some $r$-flourishing tree on $n$ nodes.
    The node $n$ cannot be appended as one of the first $r$ children of the root, so $2n-r-1$ possible positions remain.
    
    \item The second summand counts the $r$-flourishing trees that would not be $r$-flourishing after removing the node $n$.
    Node $n$ must be the only child of some other node $i$, one of the first $r$ children of the root.
    There are $n-1$ choices for $i$, and $r$ choices for its position.
    Once the nodes $n$ and $i$ are removed, an $(r-1)$-flourishing tree on $n$ nodes remains.
  \end{itemize}
  
  The formula given in the statement holds for $r=0$ or $n=0$, and a straightforward computation shows that it respects the recursive relation.
\end{proof}

Similarly to Proposition \ref{prop:decreasing-chains-linear-B}, the number of decreasing chains for the linear type $D_n$ has the following graph-theoretic interpretation.

\begin{proposition}
  Formula \eqref{eq:decreasing-chains} with $f(\pi) = s\cdot (k-1)!$ counts the increasing ordered $1$-flourishing trees on $n+1$ nodes.
  Therefore the number of maximal decreasing chains from $\zero$ to $[n]_\sigma$ for the linear type $D_n$ is $(2n-3)!!\cdot (n-1)$.
  \label{prop:decreasing-chains-linear-D}
\end{proposition}

\begin{proof}
  The first part is analogous to the proof of Proposition \ref{prop:decreasing-chains-linear-B}.
  The second part then follows from Lemma \ref{lemma:flourishing-trees} for $r=1$.
\end{proof}

\subsection{Toric arrangements}

In the toric case ($m=2$) the formulas of Theorem \ref{thm:decreasing-chains} are as follows.

\medskip
\begin{tabular}{ll}
  {\bf Type $C_n$:} & $f(\pi) = k!$. \\[0.1em]
  
  {\bf Type $B_n$:} & $f(\pi) = s \cdot (k-1)!$. \\[0.2em]
  
  {\bf Type $D_n$:} & $f(\pi) =
    \begin{cases}
      s(s-1)\cdot (k-2)! & \text{if $\pi \neq [n]$}; \\
      $1$ & \text{if $\pi = [n]$}.
    \end{cases}$
\end{tabular}
\medskip

In particular the number of decreasing chains in the toric type $C_n$ is the same as in the linear type $C_n$ (or $B_n$), and the number of decreasing chains in the toric type $B_n$ is the same as in the linear type $D_n$.

A closed formula for the toric type $D_n$ can also be found, similarly to the linear type $D_n$. One only has to separately take care of the partition $\pi = [n]$, which contributes by $(2n-3)!!$ to the total sum.

\begin{proposition}
  Formula \eqref{eq:decreasing-chains} with $f(\pi) = s(s-1)\cdot (k-2)!$ counts the increasing ordered $2$-flourishing trees on $n+1$ nodes.
  Therefore the number of maximal decreasing chains from $\zero$ to $\one$ for the toric type $D_n$ is
  \[ (2n-5)!!\cdot (n-2)(n-3) + (2n-3)!! = (2n-5)!! \cdot (n^2-3n+3). \]
  \label{prop:decreasing-chains-toric-D}
\end{proposition}

\begin{proof}
  The first part is analogous to the proof of Proposition \ref{prop:decreasing-chains-linear-B}.
  Then the second part follows from Lemma \ref{lemma:flourishing-trees} for $r=2$.
\end{proof}

\subsection{Elliptic arrangements}

In the elliptic case ($m=4$), we are going to derive a closed formula only for type $C_n$.
In order to do so, we introduce a variant of increasing ordered trees.

\begin{definition}\label{def_blooming}
  Let $q\geq 0$ be an integer. A \emph{$q$-blooming tree} on $n+1$ nodes is an increasing ordered tree on $n+1$ nodes, with $q$ extra indistinguishable unlabeled nodes (called \emph{blooms}) appended to the root.
  The only thing that matters about blooms is their position in the total order of the children of the root.
  See Figure \ref{fig:blooming-trees} for some examples.
\end{definition}

\begin{figure}[htbp]
  \begin{tikzpicture}[treestyle]
  \node [circle,draw] (1) {0}
    child {node [circle,draw,double] (a) {\phantom{a}}}
    child {node [circle,draw,double] (b) {\phantom{a}}}
    child {node [circle,draw] (2) {1}
      child {node [circle,draw] (3) {2}
      }
    };
  \end{tikzpicture}
  \qquad
  \begin{tikzpicture}[treestyle]
  \node [circle,draw] (1) {0}
    child {node [circle,draw,double] (a) {\phantom{a}}}
    child {node [circle,draw] (2) {1}
      child {node [circle,draw] (3) {2}
      }
    }
    child {node [circle,draw,double] (b) {\phantom{a}}};
  \end{tikzpicture}
  \qquad
  \begin{tikzpicture}[treestyle]
  \node [circle,draw] (1) {0}
    child {node [circle,draw] (2) {1}
      child {node [circle,draw] (3) {2}
      }
    }
    child {node [circle,draw,double] (a) {\phantom{a}}}
    child {node [circle,draw,double] (b) {\phantom{a}}};
  \end{tikzpicture}
  \caption{All different $2$-blooming trees constructed from the leftmost increasing ordered tree of Figure \ref{fig:trees} (these are $3$ of the $15$ different $2$-blooming trees on $3$ nodes).
  Blooms are unlabeled and with a double border.
  }
  \label{fig:blooming-trees}
\end{figure}
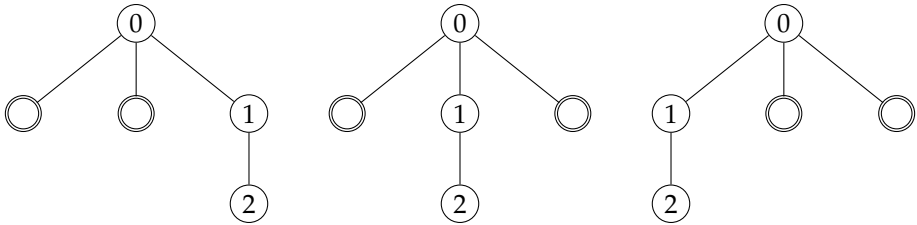

\begin{lemma}
  The number of $q$-blooming trees on $n+1$ nodes is
  \[ \frac{(2n+q-1)!!}{(q-1)!!}. \]
  \label{lemma:blooming-trees}
\end{lemma}

\begin{proof}
  The proof is by induction on $n$.
  For $n=0$ there is only one $q$-blooming tree, consisting of the root with $q$ blooms attached to it.
  Given a $q$-blooming tree on $n\geq 1$ nodes (labeled $0,\dots,n-1$), there are exactly $2n+q-1$ positions where an additional node $n$ can be attached in order to obtain a $q$-blooming tree on $n+1$ nodes.
  Every $q$-blooming tree on $n+1$ nodes is obtained exactly once, from some $q$-blooming tree on $n$ nodes.
\end{proof}

\begin{proposition}
  Formula \eqref{eq:decreasing-chains} with $f(\pi) = \frac{(k+m-2)!}{(m-2)!}$ counts the increasing ordered $\left(m-2\right)$-blooming trees on $n+1$ nodes.
  Therefore the number of decreasing chains from $\zero$ to $\one$ for the elliptic type $C_n$ is $(2n+1)!!$.
  \label{prop:decreasing-chains-elliptic-C}
\end{proposition}

\begin{proof}
  The first part is analogous to the proof of Proposition \ref{prop:decreasing-chains-linear-B}.
  Then the second part follows from Lemma \ref{lemma:blooming-trees} for $q=2$.
\end{proof}

\begin{remark}
  A closed formula for $r$-flourishing $q$-blooming trees on $n+1$ nodes (if such a formula exists at all) would yield a closed formula for the number of decreasing chains also for elliptic types $B_n$ and $D_n$.
\end{remark}

\bibliography{bibliography}
\bibliographystyle{amsalpha-abbr}

\end{document}